\documentclass[11pt]{amsart}
\usepackage{fullpage}
\usepackage[latin2]{inputenc}
\usepackage[english]{babel} 
\usepackage{tikz}
\usepackage{amsmath}
\usepackage{amsthm}
\usepackage{amssymb}
\usepackage{cite}
\usepackage{amsrefs}
\usepackage{rotating}
\usepackage{graphicx}
\usepackage{subfigure}

\usetikzlibrary{decorations.pathreplacing}

\newtheorem{thm}{Theorem}[section]

\newtheorem{conj}[thm]{Conjecture}

\newtheorem{lem}[thm]{Lemma}
\newtheorem{defi}[thm]{Definition}

\newtheorem{question}[thm]{Question}
\newtheorem{claim}[thm]{Claim}

\newtheorem{corr}[thm]{Corollary}

\theoremstyle{remark}

\begin{document}

\title{Even cycle creating paths}

\author{Daniel Soltész}
\address{MTA, Rényi Institute}
\email[Daniel Soltész]{solteszd@renyi.hu}
\thanks{\noindent This research of the author was
supported by the Hungarian Foundation for Scientific Research Grant (OTKA) No. 108947 and by the National Research, Development and Innovation Office NKFIH,  No. K-120706.}

\makeatother


\begin{abstract}
We say that two graphs $H_1,H_2$ on the same vertex set are $G$-creating ($G$-different in other papers, this difference is explained in the introduction) if the union of the two graphs contains $G$ as a subgraph. Let $H(n,k)$ be the maximal number of pairwise $C_k$-creating paths (of arbitrary length) on $n$ vertices. The behaviour of $H(n,2k+1)$ is much better understood than the behaviour of $H(n,2k)$, the former is an exponential function of $n$ while the latter is larger than exponential, for every fixed $k$. We study $H(n,k)$ for fixed $k$ and $n$ tending to infinity. The only non trivial upper bound on $H(n,2k)$ was in the case where $k=2$ $$H(n,4)\leq n^{\left(1-\frac{1}{4} \right) n-o(n)}, $$ this was proved by Cohen, Fachini and Körner. In this paper, we generalize their method to prove that for every $k \geq 2$,

$$H(n,2k) \leq n^{\left( 1- \frac{2}{3k^2-2k} \right)n-o(n)}. $$
Our proof uses constructions of bipartite, regular, $C_{2k}$-free graphs with many edges by Reiman, Benson, Lazebnik, Ustimenko and Woldar. For some special values of $k$ we can have slightly denser such bipartite graphs than for general $k$, this results in having better upper bounds on $H(n,2k)$ than stated above for these special values of $k$.
\end{abstract}

\maketitle

\section{Introduction}

The problem of determining the maximal number of pairwise $G$-creating paths on $n$ vertices has a code theoretic flavour. Indeed, we wish to have as many objects as possible (paths in this case) with the restriction that every pair of objects is different in a prescribed way (having $G$ in their union). The original motivation for these problems ultimately came from a desire to understand Shannon capacity of graphs \cite{gabortoldme}.  In previous papers \cite{original,triangle,komesi} instead of $G$-creating, the name $G$-different was used to highlight the connection with code theory. After multiple talks about the subject this name turned out to be confusing or not satisfactory for a large portion of the audiences, hence in this paper we use $G$-creating. 

Observe that in the definition of $H(n,k)$, we can safely assume that each path is of maximal length. Indeed, given a set of pairwise $C_k$-creating paths, if one of the paths $P$ is not of maximal length, we can add extra edges to it until its length reaches $n-1$. This new maximal length path was not in the original family of paths since its union with $P$ does not contain any cycle.

The study of $H(n,k)$ was initiated in \cite{komesi}. The authors of \cite{komesi} were motivated by a question concerning permutations. Hence they defined $H(n,k)$ using Hamiltonian paths of the complete graph $K_n$. They observed that the maximal number of Hamiltonian paths of $K_n$ so that every pairwise union contains an odd cycle is the number of balanced bipartitions of $[n]$. (The requirement that each union contains an odd cycle is equivalent to the requirement that no union can be bipartite. Since every Hamiltonian path is a balanced bipartite graph we cannot have more than the number of balanced bipartitions of $[n]$. And since a Hamiltonian path has a unique bipartition as a bipartite graph, any system of Hamiltonian paths with pairwise different  balanced bipartitions satisfies our conditions. ) They asked whether the answer remains the same if we insist on having a triangle in every union. This was answered affirmatively in \cite{triangle}. 

\begin{thm}[I. Kovács, D. S. \cite{triangle}]\label{thm:triangle}
For every integer $n \geq 3$,
\[H(n,3)= \begin{cases} \binom{n}{\left \lfloor \frac{n}{2} \right \rfloor} & \text{ when } n \equiv 1 \mod{2} \\
\frac{1}{2}\binom{n}{\left \lfloor \frac{n}{2} \right \rfloor} & \text{ when } n \equiv 0 \mod{2}.
\end{cases}\]
\end{thm}

Theorem \ref{thm:triangle} implies that $H(n,3)= 2^{n-o(n)}$. The hard part of Theorem \ref{thm:triangle} is the construction of a suitable set of Hamiltonian paths. The method of the construction in \cite{triangle} was generalized in \cite{oddcycle}.

\begin{thm}[I. Kovács, D. S. \cite{oddcycle}]
For every integer $ \ell \geq 1$,
$$H(n,2^{\ell}+1)= 2^{n-o(n)}. $$
\end{thm}

\noindent It is conjectured in \cite{oddcycle} that the behaviour of $H(n,2k+1)$ is similar for other values of $k$.  

\begin{conj}[\cite{triangle}]\label{conj:oddcycle}
For every integer $k \geq 1$, $$H(n,2k+1)=2^{n-o(n)}.$$
\end{conj}

\noindent The behaviour of $H(n,2k)$ is very different from  the behaviour of $H(n,2k+1)$. Upper and lower bounds for $H(n,4)$ were established in \cite{original}. 
\begin{thm}[G. Cohen, E. Fachini, J. Körner \cite{original}]\label{thm:original} For every $n$,
$$ n^{\frac{1}{2}n-o(n)} \leq H(n,4) \leq n^{\frac{3}{4}n-o(n)}.$$
\end{thm}

An easy generalization of the construction in \cite{original} gives $n^{\frac{1}{k}-o(n)} \leq H(n,2k)$. In this paper we generalize the upper bound of Cohen, Fachini and Körner for longer even cycles.

\begin{thm} \label{thm:main}
For every positive integer $k \geq 2$, 
\begin{center}
\renewcommand*{\arraystretch}{1.8}
  \begin{tabular}{l l}
  $H(n,2k)\leq n^{\left( 1- \frac{1}{k^2} \right)n-o(n)}$ & when $k=2,3,5$ \\
   $H(n,2k)\leq n^{\left( 1- \frac{2}{3k^2-2k} \right)n-o(n)}$ & when $k \neq 2,3,5$ and $k$ is even \\
   $H(n,2k)\leq n^{\left( 1- \frac{2}{3k^2-3k} \right)n-o(n)}$ & when $k \neq 2,3,5$ and $k$ is odd \\
  \end{tabular}
\end{center}
\end{thm}

Observe that the case when $k=2$ gives the upper bound of Theorem \ref{thm:original}. The reason why the upper bounds are different in the three cases is that for the proof we need the existence of bipartite, regular, $C_{2k}$-free graphs with many edges, and for different values of $k$, the order of magnitude of the number of edges for the known constructions is different, see Table \ref{results}.

The paper is organized as follows. In the short second section we present simple constructions for lower bounds on $H(n,2k)$. In Section \ref{sec:main} we prove Theorem \ref{thm:main}, in Section \ref{sec:reversing} we elaborate on the connection between $M(n,4)$ and the maximal number of pairwise reversing permutations (to be defined later). Finally in Section \ref{sec:concluding} we highlight the similarities between the proof of the present paper and the results of Cibulka, Cohen, Fachini and Körner. Also in Section \ref{sec:concluding} we elaborate on the strong connections between $M(n,2k)$ and $H(n,2k)$.

\section{Lower bound} \label{sec:lower}

In this short section we present a simple lower bound on $H(n,2k)$. 
\begin{claim}For every $k$,  $n^{\frac{1}{k}n-o(n)} \leq H(n,2k).$
\end{claim}
\begin{proof}
It is enough to construct a suitable family of Hamiltonian paths when $n \equiv 1 \mod{k}$ since $k$ is fixed and we can safely ignore a constant number of vertices. We build directed Hamiltonian paths the sole reason for this is that we can refer to the first or second etc. vertex of the path. For every $t=1,\ldots, (n-k-1)/k$, the $(tk+1)$-th vertex of every Hamiltonian path will be the vertex $tk+1 \in [n]$, we call these the \textit{fixed vertices}. Moreover, every Hamiltonian path will contain the following set of paths: $(2,3,\ldots,k), \ldots, (tk+2,tk+3, \ldots (t+1)k), \ldots (n-k+1,\ldots,n-1)$ directed towards the larger elements of $[n]$, we call these the \textit{fixed paths}. The only difference between the Hamiltonian paths will be the order of the fixed paths between the fixed vertices. For each permutation of these paths we will associate a Hamiltonian path which traverses the fixed paths in this order. (Recall that the fixed vertices have their positions fixed in every path.) Since the number of fixed paths is $(n-1)/k$, the number of Hamiltonian paths is $n^{\frac{1}{k}n-o(n)}.$ Two such paths are $C_{2k}$-creating by the following reasoning. If their permutations differ on the $i$-th coordinate then the vertices $(i-1)k$ and $ik$ are connected by a different fixed path of length $k$, in each Hamiltonian path. This results in a cycle of length $2k$ in the union. 
\end{proof}

Note that when $k=2$ the fixed paths consist of a single vertex. Also note that when $k>2$, the fixed paths of length $k-1$ can be used to enlarge the set of $C_{2k}$-creating Hamiltonian paths in the following way: If we have two sets of paths of length $k-1$ that have a $C_{2k}$ in their union, on the non fixed vertices, the two sets of paths can both be used as fixed paths. The resulting system of Hamiltonian paths will be $C_{2k}$-creating altogether. 

The author was not able to improve more than an exponential factor with this construction method. But he was not able to prove that this method can only improve an exponential factor either, see Question \ref{q:1} in Section \ref{sec:concluding}. 

\section{Upper bound} \label{sec:main}

The proof of the upper bound mimics the proof of the non-trivial upper bound by Cohen, Fachini and Körner for $C_4$-creating Hamiltonian paths  \cite{original}. Their proof takes a large set of $C_4$-creating Hamiltonian paths and produces a still large set of so called pairwise flipful permutations. (Two permutations are flipful if there are two coordinates where the two permutations have the same two elements but the order of these elements is different in the two permutations.) Then they use a theorem of Cibulka \cite{cib} to have an upper bound for the maximal number of pairwise flipful permutations. We will proceed similarly but instead of flipful permutations we will use $C_{2k}$-creating perfect matchings. (In section \ref{sec:reversing} we show that there is a connection between flipful permutations and pairwise $C_{4}$-creating perfect matchings. In Section \ref{sec:concluding} we further discuss the similarities between \cite{original}, \cite{cib} and the proof of Theorem \ref{thm:main}. ) 

\begin{defi}
Let $M(n,2k)$ be the maximal number of pairwise $C_{2k}$-creating perfect matchings of the complete graph $K_n$. 
\end{defi}

First let us establish a connection between $M(n,2k)$ and $H(n,2k)$.

\begin{lem} \label{intomatching}
For every fixed $k$, $$  n^{-(1-1/k)n-o(n)}  H(n,2k) \leq M(2n/k,2k). $$ 
\end{lem}  
\begin{proof} We first deal with the case where $n$ is even and divisible by $3k$ (we will reduce everything else to this case later). Let $\mathcal{H}$ be a set of pairwise $C_{2k}$-creating Hamiltonian paths of size $H(n,2k)$, on $n$ vertices. For every $1 \leq i \leq n$ we denote the $i$-th vertex of the Hamiltonian path $H \in \mathcal{H}$ by $\pi_H(i)$. For each Hamiltonian path $H$ we associate a triple of its subgraphs $(X_H^1,X_H^2,X_H^3)$ as follows. For every $1 \leq j \leq 3$ we define $X_H^j$ to be the induced subgraph of $H$ on the vertices $$ \bigcup_{i=0}^{n/(3k)-1} \{ \pi_H\big((j-1+3i)k+1\big),\pi_H\big((j-1+3i)k+2\big),\ldots, \pi_H\big((j-1+3i)k+k\big) \}. $$
Informally, if we partition the vertices of $H$ into consecutive subsets of size $k$, then $X_H^1$ is the induced subgraph of $H$ on the first plus the fourth plus the seventh etc. set of $k$ vertices. The useful feature of these associated triples will turn out to be that two paths with the same associated triple can only be $C_{2k}$-creating in a very specific way. The number of possible triples $(X_H^1,X_H^2,X_H^3)$ is $$\binom{n}{\underbrace{k,k,\ldots,k}_{n/k}} ((n/k)!)^{-1} \left(k! \right)^{n/k} \binom{n/k}{n/(3k),n/(3k),n/(3k)}$$
which can be seen by the following reasoning: We can partition the ground set into $n/k$ unordered parts of size $k$ in exactly $\binom{n}{k,k,\ldots,k} ((n/k)!)^{-1}$ ways, then we can choose a directed path of length $k$ in each partition in $\left(k! \right)^{n/k}$ ways, then we can partition these paths into three classes in  $\binom{n/k}{n/(3k),n/(3k),n/(3k)}$ ways. It is a routine calculation that 

$$\binom{n}{\underbrace{k,k,\ldots,k}_{n/k}} ((n/k)!)^{-1} \left(k! \right)^{n/k} \binom{n/k}{n/(3k),n/(3k),n/(3k)}  = n^{\left(1-1/k \right)n+o(n)}.$$

By the pigeon-hole principle, there is a subset $\mathcal{M}' \subseteq \mathcal{H}$ so that for every pair of Hamiltonian paths $H_1, H_2 \in \mathcal{M}'$, their associated triples are identical: $$\big(X_{H_1}^1,X_{H_1}^2,X_{H_1}^3\big)=\big(X_{H_2}^1,X_{H_2}^2,X_{H_2}^3\big)$$ and 
\begin{equation} \label{eq:pidgeon}
n^{-(1-1/k)n-o(n)}|\mathcal{H}|\leq |\mathcal{M}'|.
\end{equation}

Let $H \in \mathcal{M}'$ and let $F$ be the union of the three graphs $(X_H^1,X_H^2,X_H^3)$, thus $F$ is the disjoint union of $n/k$ directed paths on $k$ vertices. $F$ can be thought of the set of fixed edges since every edge of $F$ is contained in every Hamiltonian path in $\mathcal{M}'$. Since $\mathcal{M}'$ is a subset of $\mathcal{H}$, it consists of $C_{2k}$-creating Hamiltonian paths. We claim that in the union of any two Hamiltonian paths from $\mathcal{M}'$, no edge of $F$ is used in a cycle of length $2k$.

\begin{claim}\label{notused}
Let $H_1, H_2 \in \mathcal{M}'$ if $C$ is a (not necessarily circularly) directed cycle of length $2k$ in $H_1 \cup H_2$ then no edge of $C$ is in $F$.
\end{claim}
\begin{proof}[Proof of Claim \ref{notused}]
Recall that $F$ is a subgraph of $H_1 \cup H_2$ and $F$ is the disjoint union of $n/k$ paths, each on $k$ vertices, see Figure \ref{unionstructure}. Furthermore, every edge of $H_1 \cup H_2$ that is not in $F$, connects one endpoint of a path in $F$ from a set $X_H^j$ to a first point of a path in $F$ from a set $X_H^{j+1}$, for some $j$ (where $j+1$  is understood modulo $3$). See Figure \ref{unionstructure}. 

Suppose to the contrary that $e$ is an edge in both $C$ and $F$. Since $e$ is in $F$, it must be in $X_H^{j}$ for some $j \in \{1,2,3\}$. Thus a whole path $P(e)$ on $k$ vertices from $X_H^{j}$ must be in $C$. Therefore $C \setminus P(e)$ is a path on $k+2$ vertices that is edge disjoint from $P(e)$ and connects the endpoints of $P(e)$. Such a path must either contain an other whole path from $X_H^{j}$, or one path from $X_H^{j+1}$ and another from $X_H^{j-1}$. In the later case, $C \setminus P(e)$ contains at least $2k+2$ vertices, a contradiction. In the former case $C \setminus P(e)$ contains at least $k+4$ vertices: the two endpoints of $P(e)$, $k$ vertices from the other path in $X_H^{j}$ and two additional vertices since no edges connect two starting or two endpoints of different paths in $X_H^{j}$, a contradiction.

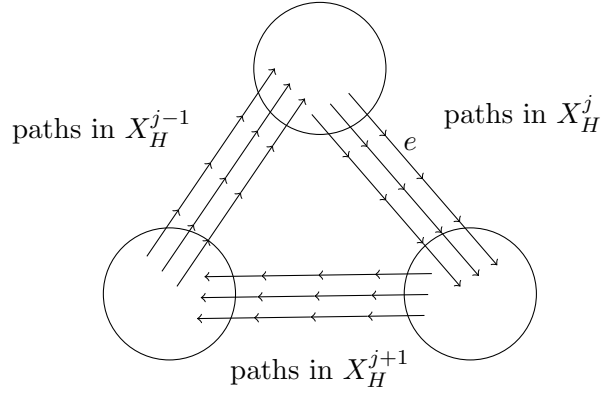
\begin{figure}

\begin{center}
\begin{tikzpicture}
\node at (3.2,2) {$e$};
\node at (4.7,2.4) {paths in $X_H^{j}$};
\node at (2,-1) {paths in $X_H^{j+1}$};
\node at (-0.9,2.2) {paths in $X_H^{j-1}$};

\draw[black] (0,0) circle (25pt);
\draw[black] (2,3) circle (25pt);
\draw[black] (4,0) circle (25pt);
\begin{scope}[shift={(-0.3,0)}]
\draw[->] (0,0.5) -- (1.7/4,0.5+2.5/4);
\draw[->] (1.7/4,0.5+2.5/4) -- (2*1.7/4,0.5+2*2.5/4);
\draw[->] (2*1.7/4,0.5+2*2.5/4)--(3*1.7/4,0.5+3*2.5/4);
\draw[->] (3*1.7/4,0.5+3*2.5/4)--(1.7,3);

\draw[->] (0.2,0.3) -- (0.2+1.7/4,0.3+2.5/4);
\draw[->] (0.2+1.7/4,0.3+2.5/4)--(0.2+2*1.7/4,0.3+2*2.5/4);
\draw[->] (0.2+2*1.7/4,0.3+2*2.5/4)--(0.2+3*1.7/4,0.3+3*2.5/4);
\draw[->] (0.2+3*1.7/4,0.3+3*2.5/4)--(0.2+4*1.7/4,0.3+4*2.5/4);

\draw[->] (0.4,0.1) -- (0.4+1.7/4,0.1+2.5/4);

\draw[->] (0.4+1.7/4,0.1+2.5/4) -- (0.4+2*1.7/4,0.1+2*2.5/4);
\draw[->] (0.4+2*1.7/4,0.1+2*2.5/4) -- (0.4+3*1.7/4,0.1+3*2.5/4);
\draw[->] (0.4+3*1.7/4,0.1+3*2.5/4) -- (0.4+4*1.7/4,0.1+4*2.5/4);
\end{scope}
\begin{scope}[shift={(1.9,2.8)},rotate=-105]

\draw[->] (0,0.5) -- (1.7/4,0.5+2.5/4);
\draw[->] (1.7/4,0.5+2.5/4) -- (2*1.7/4,0.5+2*2.5/4);
\draw[->] (2*1.7/4,0.5+2*2.5/4)--(3*1.7/4,0.5+3*2.5/4);
\draw[->] (3*1.7/4,0.5+3*2.5/4)--(1.7,3);

\draw[->] (0.2,0.3) -- (0.2+1.7/4,0.3+2.5/4);
\draw[->] (0.2+1.7/4,0.3+2.5/4)--(0.2+2*1.7/4,0.3+2*2.5/4);
\draw[->] (0.2+2*1.7/4,0.3+2*2.5/4)--(0.2+3*1.7/4,0.3+3*2.5/4);
\draw[->] (0.2+3*1.7/4,0.3+3*2.5/4)--(0.2+4*1.7/4,0.3+4*2.5/4);

\draw[->] (0.4,0.1) -- (0.4+1.7/4,0.1+2.5/4);

\draw[->] (0.4+1.7/4,0.1+2.5/4) -- (0.4+2*1.7/4,0.1+2*2.5/4);
\draw[->] (0.4+2*1.7/4,0.1+2*2.5/4) -- (0.4+3*1.7/4,0.1+3*2.5/4);
\draw[->] (0.4+3*1.7/4,0.1+3*2.5/4) -- (0.4+4*1.7/4,0.1+4*2.5/4);
\end{scope}

\begin{scope}[shift={(3.8,0)},rotate=125]

\draw[->] (0,0.5) -- (1.7/4,0.5+2.5/4);
\draw[->] (1.7/4,0.5+2.5/4) -- (2*1.7/4,0.5+2*2.5/4);
\draw[->] (2*1.7/4,0.5+2*2.5/4)--(3*1.7/4,0.5+3*2.5/4);
\draw[->] (3*1.7/4,0.5+3*2.5/4)--(1.7,3);

\draw[->] (0.2,0.3) -- (0.2+1.7/4,0.3+2.5/4);
\draw[->] (0.2+1.7/4,0.3+2.5/4)--(0.2+2*1.7/4,0.3+2*2.5/4);
\draw[->] (0.2+2*1.7/4,0.3+2*2.5/4)--(0.2+3*1.7/4,0.3+3*2.5/4);
\draw[->] (0.2+3*1.7/4,0.3+3*2.5/4)--(0.2+4*1.7/4,0.3+4*2.5/4);

\draw[->] (0.4,0.1) -- (0.4+1.7/4,0.1+2.5/4);

\draw[->] (0.4+1.7/4,0.1+2.5/4) -- (0.4+2*1.7/4,0.1+2*2.5/4);
\draw[->] (0.4+2*1.7/4,0.1+2*2.5/4) -- (0.4+3*1.7/4,0.1+3*2.5/4);
\draw[->] (0.4+3*1.7/4,0.1+3*2.5/4) -- (0.4+4*1.7/4,0.1+4*2.5/4);
\end{scope}
\end{tikzpicture}
\caption{The edges of $F$ in $H_1 \cup H_2$. All the other edges of $H_1 \cup H_2$ are in one of the three circles. The expressions $j+1$ and $j-1$ are understood modulo $3$. }
\label{unionstructure}
\end{center}
\end{figure}

\end{proof}

Let $$\mathcal{M}:= \{H \setminus F : H \in \mathcal{M}' \}. $$
$\mathcal{M}$ is a set of perfect matchings on $2n/k-2$ vertices. The union of every pair of matchings in $\mathcal{M}$ contains a cycle of length $2k$ since their original Hamiltonian paths had such a cycle in their union and by Claim \ref{notused} we only deleted edges that are not used in a cycle of length $2k$. Therefore $|\mathcal{M}|\leq M(2n/k-2,2k) \leq M(2n/k,2k)$, this combined with (\ref{eq:pidgeon}) yields

$$ n^{-(1-1/k)n-o(n)}|\mathcal{H}|\leq |\mathcal{M}'|=  |\mathcal{M}| \leq M(2n/k,2k)$$
as claimed. Thus the proof is complete when $n$ is even and divisible with $3k$.

We deal with the case where $3k$ does not divide $n$, by proving
\begin{equation}\label{eq:reducing} M(n,2k) \leq (n-1)M(n-2,2k). 
\end{equation} Since applying (\ref{eq:reducing}) at most a constant number of times, we can ensure that the ground set is even and divisible by $3k$. We prove (\ref{eq:reducing})  as follows. Let $\mathcal{M}$ be a set of pairwise $C_{2k}$-creating perfect matchings on $n$ vertices. Every perfect matchings connects the vertex $1 \in [n]$, to an other vertex from the remaining $(n-1)$ ones. By the pigeon-hole principle, there is a vertex  $i \in [n] \setminus \{1\}$ that is the neighbor of $1$ in at least $|\mathcal{M}|/(n-1)$ perfect matchings. Let $\mathcal{M}'$ be the subset of $\mathcal{M}$, that consists of those perfect matchings that connect $1$ to $i$. Observe that in the union of two perfect matchings from $\mathcal{M}'$, there must be a $C_{2k}$, since $\mathcal{M}' \subseteq \mathcal{M}$. Finally observe that in the union of two perfect matchings from $\mathcal{M}'$, $1$ and $i$ always form a connected component of size two, hence they can be deleted without destroying the $C_{2k}$-creating property. Therefore the proof is complete.

\end{proof}

Now we aim for an upper bound on $M(n,2k)$. Let $G_{PM}(C_{2k}) $ be the graph whose vertices correspond to perfect matchings on $[n]$ and two vertices of $G$ are adjacent if the corresponding perfect matchings are $C_{2k}$-creating.  Clearly $\omega(G_{PM}(C_{2k}))=M(n,2k)$. It is well known that  for every vertex transitive graph $G$, $\alpha(G) \omega(G) \leq |V(G)|$, see \cite{fractional}. (It is easy to prove that the fractional chromatic number $\chi_f (G)$  of such a graph is exactly $|V(G)|/\alpha(G)$ and clearly $\omega(G) \leq \chi_f(G)$.) Since $G_{PM}(C_{2k})$ is vertex transitive we have 
\begin{equation} 
\alpha(G_{PM}(C_{2k})) \omega(G_{PM}(C_{2k})) \leq |V(G_{PM}(C_{2k}))|
\end{equation}
or equivalently 
\begin{equation} \label{eq:vertextrans}
M(n,2k)=\omega(G_{PM}(C_{2k})) \leq \frac{|V(G_{PM}(C_{2k}))|}{\alpha(G_{PM}(C_{2k}))}.
\end{equation}

Thus we will prove an upper bound on $M(n,2k)$ by proving a lower bound to the number of pairwise non-$C_{2k}$-creating perfect matchings on $n$ vertices and using (\ref{eq:vertextrans}). We construct a large set of pairwise non-$C_{2k}$-creating perfect matchings by constructing a $C_{2k}$-free graph and proving that there are many perfect matchings in this graph. For this we will need bipartite, regular, $C_{2k}$-free graphs with many edges. 

Such constructions are often used to give lower bounds to the Turán number of even cycles. The bipartiteness and regularity properties are not required when one aims to give lower bounds to the Turán number of an even cycle. But for our method they will be essential! These constructions have an algebraic nature and they require that the number of vertices is special in some way. In Table~\ref{results} we summarize the current best constructions for bipartite, regular, $C_{2k}$-free graphs.

\begin{table}[htbp]
\begin{center}
\renewcommand*{\arraystretch}{1.8}
  \begin{tabular}{l c c c}
    \hline
  authors & 2k  & degrees & density     \\ \hline
Reiman, see below &  4  &$n^{\frac{1}{k}-o(1)}$ & $n/2=\sum_{i=0}^{2}q^i$, for $q$ a prime power \\
Benson \cite{benson} & 6  & $n^{\frac{1}{k}-o(1)}$ & $n/2=\sum_{i=0}^{3}q^i$, for $q$ a prime power  \\
Lazebnik, Ustimenko, Woldar \cite{ingredient}  &  8 & $n^{\frac{2}{3k-2}}$  & $n/2=q^{3}$  \\
Benson \cite{benson} & 10  & $n^{\frac{1}{k}-o(1)}$ & $n/2=\sum_{i=0}^{5}q^i$ , for $q$ an odd prime power  \\
Lazebnik, Ustimenko, Woldar \cite{ingredient}  & $2(2\ell)$ & $n^{\frac{2}{3k-2}}$  & $n/2=q^{2k-4-\left\lfloor  \frac{2k-3}{4}\right\rfloor}$  \\
Lazebnik, Ustimenko, Woldar \cite{ingredient}  & $2(2\ell+1)$ & $n^{\frac{2}{3k-3}}$ & $n/2=q^{2k-4-\left\lfloor  \frac{2k-3}{4}\right\rfloor}$    \\ \hline
  \end{tabular}
\end{center}
\caption{The order of magnitude of the degrees in regular, bipartite, $C_{2k}$-free graphs on $n$ vertices. The density column indicates that such constructions are known only for special ground sets. But in all cases, the set of numbers for which there are such constructions will turn out to be dense enough for all our purposes.}
\label{results}
\end{table}

We sketch the construction for the $C_4$-free case. 

\begin{claim}[\cites{reiman,cib}]
If $n/2=q^2+q+1$ then there is a bipartite, $n^{\frac{1}{2}-o(n)}$ regular, $C_{4}$-free graph on $n$ vertices. 
\end{claim}
\begin{proof}[Sketch of proof] 
A finite projective plane of order $N$ has $N^2+N+1$ points and the same number of lines. Every point is incident to $N+1$ lines and every line contains $N+1$ points. Every pair of points is contained in exactly one line. Projective planes exist when $N=q^2+q+1$ where $q$ is a prime power. A bipartite regular $C_4$-free graph can be obtained from a projective plane as follows: Let the vertices of one of the color classes be the points of the plane, the vertices of the other class be the lines of the plane. Two vertices corresponding to a point and a line are adjacent when the point is contained in the line. The graph has exactly $2N^2+2N+2$ vertices and is $N+1$ regular. This graph is $C_4$-free as every pair of vertices is contained in a single line. \end{proof}

We introduce a notation so that we can refer to the results of Table \ref{results} in a simple, unified way. 

\begin{defi}
Let $t(x)$ denote the exponent of $n$ in the third column of Table \ref{results} in the row where $k=x$.
\end{defi}

For example $t(4)=2/(3k-2)$. To show that these graphs contain many perfect matchings, we need the following results. 

\begin{thm}\label{waerden}(van der Waerden's conjecture, Gyires-Egorychev-Falikman theorem \cite{egor} \cite{falik} \cite{gyires})
If $A$ is an $n \times n$ matrix where the sum of every row and column is $1$ (a doubly stochastic matrix) then 
\[Per(A)\geq \frac{n!}{n^n}\]
where $Per(A)$ is the permanent of $A$. 
\end{thm}

\begin{lem} \label{manymatchings}
If $c$ is a constant, $r=n^{c-o(1)}$ and $G=(A,B,E)$ is an $r$-regular bipartite graph on $n$ vertices, then $G$ contains at least $n^{\frac{c}{2}n-o(n)}$ perfect matchings.
\end{lem}
\begin{proof}
In both color classes of $G$ let us fix an ordering of the vertices. Let $A$ be an $\frac{n}{2} \times \frac{n}{2}$ matrix where $a_{i,j}=1$ if and only if the $i$-th vertex of $A$ is adjacent to the $j$-th vertex of $B$. Clearly the number of perfect matchings of $G$ is equal to $Per(A)$. Let $A'=r^{-1}A$, clearly

\[Per(A)= (n^{c-o(1)})^{\frac{n}{2}} Per(A').  \]
Since the matrix $A'$ is doubly stochastic as $G$ was regular, by Theorem \ref{waerden} we have
\[Per(A)= (n^{c-o(1)})^{\frac{n}{2}} Per(A') \geq (n^{c-o(1)})^{\frac{n}{2}} \frac{(n/2)!}{(n/2)^{(n/2)}}=n^{\frac{c}{2}n-o(n)} \]
as claimed. 
\end{proof}

There are many theorems that can be used for our "the set of primes is dense enough" type argument. We (following in the footsteps of Cibulka) choose to use the most recent and most powerful one.

\begin{thm}\label{dense}(Baker-Harman-Pintz \cite{pintz})
For all large enough $n$, there is a prime in the interval $[n-n^{0.525},n].$
\end{thm}

In the next lemma we prove that although we can only construct dense, regular, bipartite, $C_{2k}$-free graphs on vertex sets of special size, these sizes are dense enough. Observe that in Table \ref{results}, for every $k$, the requirement for $n$ can be strengthened into the form: '$n=g(p)$ for some polynomial $g(x)$ and prime $p$' (not prime power!). For example, when $2k=4$, a suitable choice is $g(x)=2x^2+2x+2$. For $2k=10$, there is a single exception since there the prime $2$ cannot be used, but this only gives a single error $n=31$ which does not influence our asymptotic results. 

\begin{lem} \label{matchingbound_alpha}
Let $k$ be an integer and $g(x)=g_{k}(x)$ a polynomial for which $\lim_{x \rightarrow \infty}g(x)=\infty$. Suppose that whenever $n=g(p)$ for some prime $p$, there is a bipartite, $C_{2k}$-free, $n^{t(k)}$-regular graph on $n$ vertices. In this case, there is a family $\mathcal{M}$ of pairwise non $C_{2k}$-creating perfect matchings on $n$ vertices satisfying $n^{\frac{1}{2t(k)}n-o(n)}\leq |\mathcal{M}|.$
\end{lem}
\begin{proof} 
We say that a number $n$ is suitable when $n=g(p)$ for a prime $p$. Let $m>n_0$ be large enough for Theorem \ref{dense}, furthermore let $m$ be so large that $m-m^{0.525}$ is larger than the largest root of $g(x)$. By Theorem \ref{dense}, there is a prime $p$ in the interval $[m-m^{0.525},m]$. Since in this interval $g(x)$ is monotone increasing, there is a suitable $n$ in the interval $[g(m-m^{0.525}),g(m)]$ for every large enough $m$. 
\begin{equation} \label{eq:dense}
g(m)-g(m-m^{0.525})=o(g(m))
\end{equation} 
since for every fixed $k$, and $x$ tending to infinity $x^k-(x-x^{0.525})^k=o(x^k)$. Let now $n$ be large. Since $n$ is between $g(m)$ and $g(m+1)$ for some $m$, and $g(m+1)-g(m)=o(g(m))$ we have $n-g(m)=o(n)$. By Theorem \ref{dense}, there is a prime in the interval $[m-m^{0.525},m]$, thus there is a suitable integer in the interval $[g(m-m^{0.525}),g(m)]$. By (\ref{eq:dense}) the length of this interval is $o(g(m))=o(n)$. Therefore there is a suitable integer $n'$ such that $n'=n-o(n)$.

By our assumptions, there is a bipartite, $C_{2k}$-free, $(n')^{t(k)}$-regular graph $G$ on $n'$ vertices. By Lemma~\ref{manymatchings} there are at least $n'^{\frac{t(k)}{2}n'-o(n')}=n^{\frac{t(k)}{2}n-o(n)}$ perfect matchings in $G$. Since $n'=n-o(n)$, by adding $n-n'$ new vertices and a fixed matching on the new vertices to these, the proof is complete. 

\end{proof}

\begin{corr}\label{corr}
$$M(n,2k) \leq n^{\left(\frac{1}{2}-\frac{1}{2t(k)} \right)n-o(n)}.$$
\end{corr}
\begin{proof}
By Table \ref{results} and Lemma \ref{matchingbound_alpha} there is a set of pairwise not $C_4$-creating perfect matchings of size $n^{\frac{1}{2t(k)}n-o(n)}$ on $n$ vertices. It is well known that the number of perfect matchings on $n$ vertices is $n^{\frac{1}{2}n-o(n)}$, thus by (\ref{eq:vertextrans}) we have 
$$M(n,2k)\leq \frac{n^{\frac{1}{2}n-o(n)}}{n^{\frac{1}{2t(k)}n-o(n)}}=n^{\left(\frac{1}{2}-\frac{1}{2t(k)} \right)n-o(n)}. $$
\end{proof}

\noindent Finally we are ready to prove our main theorem. 

\begin{thm}\label{thm:main2}
For every fixed $k$ $$H(n,2k) \leq n^{\left(1-\frac{1}{k t(k)} \right)n-o(n)}. $$
\end{thm}
\begin{proof}
By Lemma \ref{intomatching} we have \begin{equation} \label{eq:main1}  n^{-\left(1-\frac{1}{k}\right)n-o(n)}  H(n,2k) \leq M(2n/k,2k). 
\end{equation}

By Corollary \ref{corr}
\begin{equation} \label{eq:main2}
M(2n/k,2k) \leq (2n/k)^{\left(\frac{1}{2}-\frac{1}{2t(k)} \right)2n/k-o(n)}=n^{\left(\frac{1}{k}-\frac{1}{k t(k)} \right)n-o(n)}.
\end{equation}

Equations (\ref{eq:main1}) and (\ref{eq:main2}) together yield the claimed upper bound.
\end{proof}

Theorem \ref{thm:main} follows from Theorem \ref{thm:main2} and Table \ref{results}.

\section{Connection with reversing permutations} \label{sec:reversing}

\begin{defi}
We say that two permutations $\pi_1 , \pi_2$ of the elements $[n]$ are reversing if there are two coordinates $1 \leq i < j \leq n$ for which $\pi_1(i)=\pi_2(j)$ and $\pi_1(j)=\pi_2(i)$. Let $RP(n)$ denote the maximal number of pairwise reversing permutations of $[n]$. 
\end{defi}

In this short section we establish a connection between $M(n,4)$ and $RP(n/2)$. In \cite{original} the authors prove Theorem \ref{thm:original} using a relation between $RP(n/2)$ and $H(n,4)$, for more details see Section \ref{sec:concluding}. The following lemma states that ignoring exponential factors, the values $M(n,4)$ and $RP(n/2)$ are the same. 

\begin{claim} \label{claim:permcorrespondence} When $n$ is even, 
 $$2^{\frac{n}{2}}\binom{n}{\frac{n}{2}}^{-1}M(n,4) \leq RP\left(\frac{n}{2}\right) \leq M(n,4). $$
\end{claim}
\begin{proof}
 For a permutation $\pi$ of the elements $[n/2]$, let us associate the perfect matching $M(\pi)$ on the vertices $[n]$ that consists of the edges $(i,\pi(i)+n/2)$. Observe that two permutations of $[n/2]$ are reversing if and only if their associated matchings are $C_4$-creating. This proves the second inequality. Observe that this correspondence is a bijection between the set of permutations of $[n/2]$ and the set of perfect matchings on $[n]$ which have all of their edges between the sets $\{1, \ldots n/2\}$ and $\{n/2+1, \ldots n\}$.  The first inequality follows from the observation that given a set of $M(n,4)$ pairwise $C_4$-creating perfect matchings on $n$ vertices, the average number of perfect matchings that have their edges between $S \subset [n]$ and $[n] \setminus S$, averaging over every $|S|=n/2$ is $2^{\frac{n}{2}} \binom{n}{n/2}^{-1}M(n,4)$. And we can have a similar bijection between these perfect matchings and permutations of $[n/2]$ that takes a pair of $C_4$-creating perfect matchings into a pair of reversing permutations. 
\end{proof}

Thus from the proof of Claim \ref{claim:permcorrespondence} we see that $RP(n/2)$ can be viewed as a version of $M(n,4)$ where we restrict our matchings to have all their edges between two fixed subsets of $[n]$.

\section{Concluding remarks}\label{sec:concluding}

It might not be immediately apparent that the proof of Theorem \ref{thm:main2} in the case when $k=2$ is essentially equivalent to the proof of Theorem \ref{thm:original}. Let us elaborate on this equivalence. The proof of Theorem \ref{thm:main2} follows the following steps.
\begin{itemize}
\item Lemma \ref{intomatching} establishes a connection between $H(n,2k)$ and $M(n,2k)$.
\item (\ref{eq:vertextrans}) gives a very rough upper bound on $M(n,2k)$ using a lower bound on $\alpha(G_{PM}(C_{2k}))$.
\item We give a lower bound on $\alpha(G_{PM}(C_{2k}))$ using $C_{2k}$-free graphs that contain many perfect matchings. 
\end{itemize}
The essential equivalence in the $k=2$ case can be seen as follows. In \cite{original} the authors establish a connection between $H(n,4)$ and $RP(n/2)$ (recall that by Claim \ref{claim:permcorrespondence} we already know that $RP(n/2)$ is only an exponential factor away from $M(n,4)$). Their proof is generalized to Lemma \ref{intomatching}. Then the authors of \cite{original} refer to the upper bound on $RP(n/2)$ proved in \cite{cib} to conclude that $H(n,4) \leq n^{\frac{3}{4}n-o(n)}$. In \cite{cib} the author actually proves that the maximal number of pairwise non-reversing permutations of $[n/2]$ is equal to $n^{\frac{1}{4}n-o(n)}$. Then he uses (\ref{eq:vertextrans}) to prove an upper bound on $RP(n/2)$. Observe that for the upper bound on $RP(n/2)$ (and thus $H(n,4)$) we only need the lower bound on the number of pairwise non-reversing permutations! (The proof of the upper bound is much longer and much more difficult.) Our lower bound on the number of pairwise non $C_{2k}$-creating perfect matchings is a natural generalization of the lower bound in \cite{cib}. Therefore the main result of the present paper (Theorem \ref{thm:main2}) should be considered a natural generalization of the ideas of Cibulka, Cohen, Fachini and Körner which led to Theorem \ref{thm:original}.

In \cite{cib}, it is proven that the maximal number of pairwise non-reversing permutations of $[n/2]$ is $n^{\frac{1}{4}n-o(n)}$. Therefore using only \ref{eq:vertextrans} we cannot get a smaller upper bound on $M(n,4)$ than $n^{\frac{1}{4}n-o(n)}$. Or in other terms, the fractional clique number of $G_{PM}(C_{4})$ is $n^{\frac{1}{4}n-o(n)}$ (since for vertex transitive graphs $G$ the fractional clique number is $|V(G)|/\alpha(G)$) thus if the clique number is actually smaller, no method can prove it which would also work for the fractional clique number. 

We saw in Section \ref{sec:lower} that $H(n,2k)$ is larger than any exponential function of $n$, for every fixed $k$. In the constructions presented there, for every $k>2$, every Hamiltonian path contains a set of roughly $\frac{n}{k}$ paths on exactly $k-1$ vertices. Since every Hamiltonian path constructed there contains these fixed paths, we can add any Hamiltonian path which forms a $C_{2k}$ with these fixed paths. It is natural to try to add a set of Hamiltonian paths that is constructed similarly but the set of paths of length $k-1$ is different, moreover $C_{2k}$-creating from the original set of fixed paths. Let $\mathcal{P}_k(n)$ be the set of graphs on $n$ vertices that is the disjoint union of paths of length $k-1$. This is the motivation behind the following questions.

\begin{question}\label{q:1}
Let $k>2$. What is the maximal number of pairwise $C_{2k}$-creating graphs from $\mathcal{P}_k(n)$? Is the answer an exponential function of $n$?
\end{question}

The first non-trivial case of Question \ref{q:1} is: what is the maximal number of pairwise $C_6$-creating perfect matchings on $n$ vertices? Or in other words, what is the value of $M(n,6)$?  Although for larger $k$, $\mathcal{P}_k(n)$ contains longer paths it is not hard to see that for $k>3$, a larger than exponential lower bound for $M(n,2k)$ would result in larger than exponential lower bound for the number of pairwise $C_{2k}$-creating graphs from $\mathcal{P}_k(n)$. This, and Lemma \ref{intomatching} means that for $k$ at least $3$, better lower bounds for $M(n,2k)$ would lead to better constructions for $H(n,2k)$, and better upper bounds for $M(n,2k)$ would lead to better upper bounds on $H(n,2k)$.

\section{acknowledgement}

The author would like to thank Gábor Simonyi and Kitti Varga for their valuable comments and suggestions that improved the quality of this manuscript.

\end{document}